\documentclass[12pt,letterpaper]{amsart}
\usepackage{amsmath, amsthm, amssymb, mathrsfs}
\usepackage[tmargin=1.35in,bmargin=1.2in,rmargin=1.4in,lmargin=1.4in]{geometry}
\usepackage[breaklinks=true]{hyperref}
\usepackage{tikz-cd}

\theoremstyle{plain}
\newtheorem{theorem}{Theorem}[section]

\newtheorem{lemma}[theorem]{Lemma}
\newtheorem{prop}[theorem]{Proposition}

\theoremstyle{definition}
\newtheorem{defin}[theorem]{Definition}
\newtheorem{remark}[theorem]{Remark}

\DeclareMathOperator{\gr}{\ensuremath{gr}}

\DeclareMathOperator{\im}{\ensuremath{im}}
\DeclareMathOperator{\Mod}{\ensuremath{-Mod}}
\DeclareMathOperator{\Hom}{\ensuremath{Hom}}

\DeclareMathOperator{\Pic}{\ensuremath{Pic}}
\DeclareMathOperator{\Aut}{\ensuremath{Aut}}
\DeclareMathOperator{\Sym}{\ensuremath{Sym}}
\DeclareMathOperator{\Id}{\ensuremath{Id}}

\begin{document}
\title{Morita equivalence problem for symplectic reflection algebras}
\author{Akaki Tikaradze}
\email{Akaki.Tikaradze@utoledo.edu}
\address{University of Toledo, Department of Mathematics \& Statistics, 
Toledo, OH 43606, USA}

\begin{abstract}
In this paper we fully solve the Morita equivalence problem for 
symplectic reflection algebras associated to direct products of finite subgroups of $SL_2(\mathbb{C})$. Namely given  a pair of such symplectic reflection algebras $H_c, H_{c'}$,
 then $H_c$ is Morita equivalent to $H_c'$ if and only if they are related by
 a standard Morita equivalence. We also establish new cases for Morita classification problem for type A rational Cherednik algebras. Our approach crucially relies on the reduction modulo large primes.
\end{abstract}

\maketitle

\section{Introduction}

 Symplectic reflection algebras were introduced by Etingof and Ginzburg \cite{EG}. Let us recall their definition.
 Let $V$ be a symplectic vector space over $\mathbb{C}$ with the symplectic form $\omega$, let $\Gamma\subset Sp(V)$ be a finite subgroup.
 Then an element $\gamma\in\Gamma$ is said to be a symplectic reflection if $\im(Id-\gamma)=2$. Denote by $\omega_{\gamma}$ the restriction of $\omega$ on
 $\im(Id-\gamma).$ Let $S\subset\Gamma$ be the set of all symplectic reflections, then $S$ is $\Gamma$-conjugation invariant.
  Let $c:S\to\mathbb{C}$ be a $\Gamma$-invariant function, and $t\in\mathbb{C}.$ Then the symplectic reflection algebra
  $H_{t, c}(\Gamma, V)=H_{t, c}$ corresponding to parameters $(t, c)$ is defined as the quotient of 
  the smash-product algebra $T(V)\# \Gamma$
  by the following relations
  $$[v, w]=t[v, w]+\sum_{\gamma\in \S}w_{\gamma}(v, w)\gamma,\quad v, w\in V.$$
The spherical subalgebra $B_{t, c}$ of $H_{t, c}$ is defined as $eH_{t,c}e$, where $e=\frac{1}{|\Gamma|}\sum_{g\in\Gamma}g.$
From now on we put $t=1$ and abbreviate $H_{1, c}$ (respectively $B_{1, c}$) by $H_c$ (resp. $B_c).$
Both algebras $H_c, B_c$  can be equipped with a standard filtration and their PBW property states that 
$$\gr(H_c)=\Sym(V)\#\Gamma,\quad \gr(B_c)=\Sym(V)^{\Gamma}.$$
Thus, algebras $B_c$ can be thought of as noncommutative deformations of a singular Poisson variety $V/\Gamma.$
Also, algebras $H_c$ represent multiparameter deformations of the smash product $A(V)\#\Gamma$, where $A(V)$ denoted the Weyl algebra generated by $V.$

For a fixed pair $V, \Gamma\leq Sp(V)$ as above, it is an interesting problem to classify all Morita equivalent pairs of symplectic reflection algebras
$H_c, H_{c'}.$ To the best of our knowledge, the following (short) list of results is a full one in this direction. If $\dim (V)=2$ and $\Gamma=\mathbb{Z}/2\mathbb{Z}$,
then algebras $B_c$ are isomorphic to maximal primitive quotients of $U(\mathfrak{sl}_2)$, in which case the Morita equivalence problem 
was solved by T. Hodges \cite{H1}. In the case of the symmetric group $\Gamma=S_n$ and $V=\mathfrak{h}\oplus\mathfrak{h}^*$ (where $\mathfrak{h}$ is the standard $n-1$-dimensional
representation), the problem was solved by Berest, Etingof and Ginzburg \cite{BEG} for very generic values of parameter $c$ (algebraically independent over 
$\mathbb{Q}$).
Some partial results about Morita classification of classical generalized Weyl algebras
were obtained in \cite{RS}
Finally, for $\dim(V)=2$ and an arbitrary $\Gamma$, very generic cases of the Morita classification problem was resolved in \cite{T}.

For $\dim(V)=2$ and $\Gamma\leq SL_2(\mathbb{C})$, corresponding symplectic reflection algebras along with their spherical subalgebras (Noncommutative deformations of Kleinian singularities) were introduced 
by Crawley-Boevey and Holland
\cite{CBH}. A fundamental feature of their work was the introduction of deformed preprojective algebras of quivers and an isomorphism between $H_{c}$ and
the corresponding deformed preprojective algebras of the affine Dynkin quiver corresponding to $\Gamma$ via the Mckay correspondence.
A key tool in their study of deformed projective algebras were certain reflection functors which induce Morita equivalence between algebras with different parameters.

The following is the main result of this paper.

\begin{theorem}\label{main}
Let $V=\bigoplus_{i=1}^nV_i$ with $V_i=\mathbb{C}^2$ and $\Gamma_i\leq SL_2(\mathbb{C}).$
Let $\Gamma=\prod_i^n\Gamma_i\leq Sp(V).$ Let $U$ denote the space of parameters corresponding
to symplectic reflection algebras of the pair $(\Gamma, V).$ Denote by $W_{ext}$ the subgroup of affine transformations
of $U$ generated by reflections, diagram automorphisms and conjugations by the normalizer of $\Gamma$ in $Sp(V).$
Then $H_{c}$ is Morita equivalent to $H_{c'}$ if and only if there exists $\gamma\in W_{ext}$ such that
$c'=\gamma(c).$


\end{theorem}

Our result has a particularly simple formulation for type A case. Noncommutative deformations of type A singularities
are also known as 
(classical) generalized Weyl algebras, and as such they were introduced and studied by V.~Bavula,
T.~Hodges and others. Let us recall their definition.
Let $v\in \mathbb{C}[h]$.
Then the corresponding (classical) generalized Weyl algebra $A(v)$ is defined as a
quotient of $\mathbb{C}\langle x, y, h\rangle$
by the following relations
\[
[h, x]=x, \quad [h, y]=-y, \quad xy=v(h), \quad yx=v(h-1).
\]
The Morita classification problem for algebras $A(v)$ was posed and initially studied by T. Hodges.
The following result positively answers the question raised by him in \cite{H2}.

\begin{theorem}\label{GWA}

Let $v\in\mathbb{C}[h]$ be a monic polynomial of degree $n$ with distinct roots $t_1,\cdots, t_n.$
Let $v'\in\mathbb{C}[h]$ be a monic polynomial with distinct roots $t'_1,\cdots, t'_n.$
Then $A(v')$ is Morita equivalent to $A(v)$ if and only if there exist $\epsilon=\pm 1, \sigma\in S_n, c\in\mathbb{C}$
and $d\in\mathbb{Z}^n$ so that 
$$t'_i=\epsilon t_{\sigma(i)}+d_i+c,\quad 1\leq i\leq n.$$
\end{theorem}

\section{Reflection functors and their action on $K_0$}

We start by briefly recalling the definition of deformed preprojective
algebras introduced by Crawley-Boevey and Holland \cite{CBH}. Let $Q$ be
an arbitrary finite quiver with the set of vertices $I$, and let
$\overline{Q}$ be its double. For each $i\in I$, denote by $e_i$ the
corresponding idempotent. Let $K$ be a field and $\lambda\in K^I.$ Then the deformed
preprojective algebra $\Pi^{\lambda}(Q)$ with parameter $\lambda$ is
defined as the quotient of the path algebra of $\overline{Q}$ over $K$ by the
following relations:
\[
\sum_{a\in Q}[a, a^*]=\sum_{i\in I} \lambda_ie_i.
\]

Next, we need to recall the reflection functors (for $\lambda_i\neq 0$)
\[
E_i:\Pi^{\lambda}(Q) \Mod \to \Pi^{r_i(\lambda)}(Q) \Mod
\]
satisfying $E_i^2={\rm Id}$, introduced by Crawley-Boevey and Holland
\cite{CBH}. First, they define the symmetrized Ringel form $(\cdot,\cdot)
: \mathbb{Z}^I \times \mathbb{Z}^I \to \mathbb{Z}$, via:
\[
(\alpha, \beta) = \langle \alpha, \beta \rangle + \langle \beta, \alpha
\rangle, \quad \text{where} \quad
\langle \alpha, \beta \rangle := \sum_{i \in I} \alpha_i \beta_i -
\sum_{a \in Q} \alpha_{t(a)} \beta_{h(a)}.
\]
Next, define a simple reflection for a given loop-free vertex $i\in I$ as
follows:
\[
s_i:\mathbb{Z}^{I}\to\mathbb{Z}^I, \quad s_i(\alpha)=\alpha-(\alpha,
\epsilon_i)\epsilon_i,
\]
where $\epsilon_i$ is the $i$-th coordinate vector, and the corresponding
reflection $r_i:K^I\to K^I$ defined with the property
that $r_i(\lambda)\cdot\alpha=\lambda\cdot s_i(\alpha),$ explicitly 
\[
r_i(\lambda)_j=\lambda_j-(\epsilon_i, \epsilon_j)\lambda_i.
\]
One now defines the Weyl group $W$ of $Q$ as the group of automorphisms of
$\mathbb{Z}^I$ generated by $s_i, i\in I.$
Note that $W$ acts on $K^I$ by
$w(\lambda)\cdot\alpha=\lambda\cdot w^{-1}(\alpha)$
for all $w\in W, \lambda\in K^I, \alpha\in \mathbb{Z}^{I}.$

We use the following standard terminology: $\lambda\in\mathbb{C}^I$ is said to be regular if $\lambda\cdot\alpha\neq 0$ for any Dynkin root
$\alpha$, while $\lambda$ is called generic if $\lambda\cdot\alpha\neq 0$ for any root $\alpha.$

Given a left $\Pi^{\lambda}$-module $M$, we put $M_i=e_iM$, where recall
that $e_i\in\Pi^{\lambda}$ denotes the idempotent of the vertex $i\in I.$
So, $M=\bigoplus_{i\in I}M_i.$ Then $(E_i(M))_j=M_j$ for $j\neq i$ and
$(E_i(M))_i$ is a direct summand of $\bigoplus_{a\in Q, h(a)=i}
M_{t(a)}.$ When $\lambda_i=0$ then we put $E_i=Id.$
By composing reflection functors, for any element of the Weyl group $w\in
W$, we get the functor, which is an equivalence of categories
\[
E_w:\Pi^{\lambda} \Mod \to \Pi^{w(\lambda)} \Mod.
\]

Now we recall how symplectic reflection algebras $H_{\lambda}$ (for $\dim(V)=2$) relate to
deformed preprojective algebras. 
So, $V=\mathbb{C}^2$ and $\Gamma$ a finite
subgroup of $SL_2(\mathbb{C}).$
Let $\lambda\in Z(\mathbb{C}[\Gamma]).$ 
Then
the corresponding symplectic reflection algebra is defined as follows:

$$H_{\lambda}=(\mathbb{C}\langle x, y\rangle\#\Gamma)/([x,
y]-\lambda).$$

Then the corresponding spherical subalgebras $\mathcal{O}^{\lambda}=eH_{\lambda}e$
(recall that $e=\frac{1}{|\Gamma|}\sum_{g\in\Gamma}g$) were introduced by Crawley-Boevey and Holland as noncommutative deformations
of Kleinian singularities.

Let $\mathbb{C}^I$ denote the space of parameters of $\lambda,$ so $I$ consists of vertices of the McKay quiver of $\Gamma$-isomorphism classes of irreducible $\Gamma$-representations.
Let $Q$ be the McKay graph of $\Gamma$
with arbitrary orientation. Thus vertices of $Q$ correspond to simple
representations $V_i$ of $\Gamma$ (the trivial module corresponds to
vertex 0), and with an arrow $i\to j$ if $\Hom_{\Gamma}(V\otimes V_i,
V_j)\neq 0.$ 
Then, an important result of \cite{CBH} asserts that 
$\Pi^{\lambda}(Q)$ is Morita equivalent to $H_{\lambda}$, and
$\mathcal{O}^{\lambda}\cong e_0\Pi^{\lambda}e_0.$
Here we first write $\lambda\in\mathbb{C}^I$ as
$\lambda=\sum_{i \in I} f_i \tau_i$, where
$\tau_i$ is the idempotent in
$Z(\mathbb{C}[\Gamma])$ corresponding to the simple $\Gamma$-module
$V_i.$ Now $\lambda$ has $i$-th entry $f_i\delta_i$, where $\delta_i :=
\dim(V_i)$.
Moreover, the algebras $\Pi^{\lambda}$ and $\mathcal{O}^{\lambda}$ are
Morita equivalent if $\lambda\cdot \alpha\neq 0$ for all Dynkin roots
$\alpha.$

From now on, we say that $\lambda\in\mathbb{C}^{I}$ is generic if
$\lambda\cdot\alpha\neq 0$ for any root $\alpha.$ Similarly,
$\lambda\in\mathbb{C}^{I}$ is said to be very generic if its coordinates
are algebraically independent over $\mathbb{Q}.$ Recall that
$\mathcal{O}^{\lambda}$ is commutative if and only if
$\lambda\cdot\delta=0$, and  $\text{gl.dim}(\mathcal{O}^{\lambda})=1$ if
and only if $\lambda$ is generic \cite[Theorem 0.4]{CBH}.

Let $W_{\text{ext}}\leq \Aut(\mathbb{Z}^I)$ be the group generated by the reflections $s_i$ and
the diagram automorphisms of $Q$. We will also denote by
$W_{ext}$ the subgroup of $\Aut(\mathbb{C}^I)$ generated by reflections $r_i$ and the diagram automorphisms.
Then it is well-known that $W_{\text{ext}}$ can be identified
with the extended
affine Weyl group of the affine root system of $Q.$ Hence given any
$\lambda \in \mathbb{C}^{I},$ for any $w\in W_{\text{ext}}$, we have the
corresponding equivalence of categories
\[
\bar{E}_{w}:\Pi^{\lambda}-mod \to \Pi^{w(\lambda)}-mod.
\]
 We also need to recall that as
observed by Boyarchenko \cite{B}, translations by 
\begin{equation}\label{Translations}
d\in\Lambda := \lbrace \xi\in \mathbb{Z}^I,\
\xi \cdot \delta = 0 \rbrace
\end{equation}
belong to $W_{ext}$ and hence give rise to Morita equivalences

\[
E_d:\Pi^{\lambda}-mod \to\Pi^{\lambda+d}-mod.
\]

Let $K$ be a field.
Given a finite $K$-dimensional $\Pi^{\lambda}$-module $V$, 
we denote by $\dim_j(V)$ the dimension of $V_j=e_jV.$ So we have its dimension vector $\dim(V)=(\dim_j(V), j\in I)\in \mathbb{Z}^{I}.$
It is well-known that $\dim(E_i(V))=s_i(\dim(V))$ as long as $\lambda_i\neq 0.$ The next result, which should be well-known, extends this on the level
of Grothendieck groups.
Recall that as long as char$(K)$ is large enough, then $K_0(\Pi^{\lambda})\cong \mathbb{Z}^I$ with the basis given by the classes of projective modules $\Pi^{\lambda}e_i, i\in I$. Thus we may and
will identify $K_0(\Pi^{\lambda})$ with $\mathbb{Z}^I.$
 We include the proof for the reader's convenience.
\begin{lemma}\label{K_0}
Let $K$ be a field with char$(K)=0$ or $\gg 0$. Let $\lambda\in K^I$ be so that $\lambda_i\neq 0, i\in I$.
Then $K_0(E_i): \mathbb{Z}^I\to \mathbb{Z}^I$ equals to the reflection $s_i.$
\end{lemma}

\begin{proof}
$We$ may assume without loss of generality that $K=\mathbb{C}.$ 
By $\Pi$ we denote the global version of $\Pi^\lambda$ defined over the ring of Laurent
polynomials $R=\mathbb{C}[\lambda_j, j\in I, \lambda_i^{-1}].$ Also, denote by $\Pi'$ the global version of
$\Pi^{r_i(\lambda)}.$
Then we have the global reflection functors which are mutual inverses
 $$E_i:\Pi-mod\to \Pi -mod,\quad E'_i:\Pi'-mod\to\Pi-mod,$$ and we know that
$K_0(\Pi)=K_0(\Pi^{\lambda})$ is a free abelian group with  a basis $\lbrace [\Pi e_j], j\in I\rbrace.$
Let 
$$K_0(E_i)(\Pi e_j)=\sum_{k\in I} x_k[\Pi e_k]\in K_0(\Pi'),\quad x_k\in\mathbb{Z}.$$ 
Put $x=(x_k)\in\mathbb{Z}^{I}.$ We want to show that $x=s_i(\epsilon_j).$
Let $V$ be a finite dimensional $\Pi'$-module. We have 
$$\dim(Hom_{\Pi'}(E_i(\Pi e_j), V))=\dim\Hom_{\Pi}(\Pi e_j, E_i(V))=\dim_j(E_i(V)).$$
We know that $dim(E_i(V))=s_i(\dim(V)).$ Thus, $(x-s_i(\epsilon_j))\cdot \dim(V)=0$ for any finite dimensional
$\Pi$-module $V.$ Since dimension vectors of all finite dimensional $\Pi$-modules span $\mathbb{Z}^I,$
we are done.
\end{proof}

\begin{lemma}
Let $K$ be a field with char$(K)=0$ or $\gg 0$. Let $\lambda, \lambda'\in K^I$
Let $w\in W_{\text{ext}}$ be such that $w(\lambda)=\lambda'$. Then there exists $\theta\in W_{\text{ext}}$
so that $\lambda'=\theta(\lambda)$ and $K_0(E_{\theta})=\theta.$

\end{lemma}
\begin{proof}
Write $w=\sigma w'$, where $\sigma$ is a diagram automorphism, and $w'\in W$ has the minimal possible length among all $w'\in W$ so that
$\lambda'=\sigma w'(\lambda).$ Put $\theta=\sigma w'$ Now we get an equivalence 
$$E_{w}:\Pi^{\lambda}-mod\to \Pi^{\lambda'}-mod$$
and $K_0(E_w)=w.$

\end{proof}
The next result is a key one.
\begin{prop}\label{Mitya}
Let $p\gg 0$ be a prime.
Let $\lambda, \lambda'\in \mathbb{F}_p^I$ so that $\lambda\cdot\rho=\lambda'\cdot\rho=1.$
Then there exists a Morita equivalence $E: \Pi^{\lambda}-mod\to \Pi^{\lambda'}-mod$ so that $K_0(E)\in W_{\text{ext}}.$

\end{prop}
\begin{proof}
It suffices to show by the above lemma that there exists $w\in W_{\text{ext}}$, so that $w(\lambda)=\lambda'.$
Recall that $\Lambda=\lbrace x\in\mathbb{Z}^I| x\cdot\rho=0.\rbrace$.
Recall also that translations by elements of $\Lambda$ belong to $W_{ext}$.
Let $e_1,\cdots, e_n$ be a $\mathbb{Z}$-basis of $\Lambda.$
Given a prime $p$, put 
$$\Lambda_p=\Lambda_p=\lbrace x\in\mathbb{F}_p^I| x\cdot(\rho\mod p)=0\rbrace.$$
Then for $p\gg 0,$ elements $e_1\mod p,\cdots, e_n\mod p$ form a $\mathbb{F}_p$-basis of $\Lambda_p.$
Hence, it follows that for any $\gamma\in\Lambda_p$ there exists $\alpha\in W_{ext}$ so that
the translation by $\gamma$ on $\mathbb{F}_p^I$ equals to the action by $\alpha.$
Now, since $\lambda'-\lambda\in \Lambda_p,$ there exists $w\in W_{ext}$, so that $w(\lambda)=\lambda'$ and we are done.

\end{proof}

\section{The action of the Picard group on $K_0$}
Let $\bold{k}$  be an algebraically closed field of characteristic $p>2.$
In this section,  we fix a symplectic $\bold{k}$-vector space $V$ and a finite subgroup $\Gamma\leq Sp(V)$
such that $p$ does not divide $|\Gamma|.$ Just as before, we identify $K_0(A(V)\#\Gamma)$ with $K_0(\bold{k}[\Gamma])=\mathbb{Z}^I,$
where $I$ is the set of isomorphism classes of irreducible $\Gamma$-representations.

\begin{defin}\label{diagram}
    In the above setting, we call $A\in \Aut(\mathbb{Z}^I)$ a diagram automorphism of the pair  $(V, \Gamma)$  if there exist $t\in Sp(V)$ in the normalizer of $\Gamma$ 
    and a character
$\chi:\Gamma^*$ so that $A=K_0(t\tau_{\chi}),$ where $t$ is viewed as an automorphism of $A(V)\#\Gamma$ in the natural way, and
$\tau_{\chi}\in\Aut(A(V)\#\Gamma)$ is the twist by character $\chi$ defined as follows: $\tau_{\chi}(g)=\chi(g)g$ for $g\in\Gamma$,
and $\tau_{\chi}(v)=v$ for $v\in A(V).$
\end{defin}
It follows easily that given a diagram automorphism $\sigma\in \Aut(\mathbb{Z}^I)$,
then $\sigma$ is a permutation of $I.$ Moreover, given simple $\Gamma$-representations 
$V_i, V_j, i, j\in I$,
we have $$\Hom_{\Gamma}(V\otimes V_i, V_j)=\Hom_{\Gamma}(V\otimes V_{\sigma(i)}, V_{\sigma(j)}).$$  
Therefore $\sigma$ is a diagram automorphism of the Mckay quiver of $\Gamma$ (for the case $\dim V=2).$

We will identify the center of $A(V)\#\Gamma$ with the Frobenius twist of $\mathcal{O}(V/\Gamma)$ and as such it is equipped with the natural Poisson bracket.

Recall that given a $\bold{k}$-algebra $A,$  its Picard group $\Pic(A)$  is defined as the group of $\bold{k}$-linear auto-equivalences of the category of $A$-modules.
It is well-known that  $\Pic(A)$  can be identified with the group of isomorphism classes of invertible $A$-bimodules.
The following result will play a key role in the proofs of our many results.

\begin{theorem}\label{key}

Let $f: \Pic(A(V)\#\Gamma)$ be  such that on the level of the center $Z(f)$ preserves the Poisson bracket.
Then $K_0(f)\in \Aut(\mathbb{Z}^I)$ is a diagram automorphism of the pair  $(V, \Gamma).$ 

\end{theorem}
\begin{proof}


Denote by $\hat{V}$ the formal neighbourhood of $V$ at the origin. Similarly,
$A(\hat{V})$ denotes the completion of $A(V)$ at the origin.
Recall that $A(V)$ (and hence $A(\hat{V})$) is an Azumaya algebra over the Frobenius twist of $V$ (respectively of $\hat{V}$). 
 It was shown by Bezrukavnikov and Kaledin \cite{BK}
that $A(\hat{V})$ as an Azumaya algebra admits a $\Gamma$-equivariant splitting, which gives rise to a Morita equivalence
$$i: A(\hat{V})\#\Gamma-mod\to \mathcal{O}(\hat{V})\#\Gamma-mod.$$ Moreover, once we
identify $K_0(A(\hat{V})\#\Gamma)$ and $K_0(\mathcal{O}(\hat{V})\#\Gamma)$ with $K_0(\bold{k}[\Gamma])=\mathbb{Z}^I,$
then $K_0(i)=p^l\Id_{\mathbb{Z}^I},$ where $l=\frac{1}{2}\dim V$. So, it suffices to show that given $F \in\Pic(\mathcal{O}(\hat{V})\#\Gamma)$ such that
$Z(F):\mathcal{O}(\hat{V}/\Gamma)\to\mathcal{O}(\hat{V}/\Gamma)$ is a Poisson automorphism,
then $K_0(F)$ is a diagram automorphism.

At first, we show that for any Poisson automorphism $f\in\Aut(\hat{V}/\Gamma),$ there exists
$\theta\in \Aut(\mathcal{O}(\hat{V})\#\Gamma)$ so that $Z(\theta)=f$ and $K_0(\theta)$ is a diagram automorphism.
Denote by $U$ the smooth locus of $\hat{V}/\Gamma$, and let $W$ be the pre-image of $U$ in $\hat{V}$ under the quotient
map $\pi:\hat{V}\to \hat{V}/\Gamma.$ So
$\pi:W\to U$ is a $\Gamma$-\'etale covering. Hence, $f\circ\pi: W\to U$ is also a $\Gamma$-\'etale covering. Now, since $W$ admits no nontrivial
$p'$-etale covering (as the compliment of $W$ in the affine space $\hat{V}$ is of codimension $\geq 2$),
it follows that exists an \'etale map $\psi: W\to W$ lifting $f|_U.$ So, $\psi$ is a Poisson isomorphism and normalizes the image of $\Gamma$ in $\Aut(W).$
Let $\sigma\in \Aut(\Gamma)$ be so that $\psi\gamma\psi^{-1}=\sigma(\gamma), \gamma\in\Gamma.$
Since $\mathcal{O}(\hat{V})=\mathcal{O}(W)$ (once again because codimension of $\hat{V}\setminus W\geq 2$), we conclude that $\psi\in \Aut(\hat{V})$ 
is a symplectic (non-linear) automorphism of $\hat{V}.$ 
Let $m$ be the maximal ideal of the origin in $\mathcal{O}(\hat{V}).$ Then we may identify $m/m^2$ with $V$ as a $\bold{k}$-vector space.
 Let $h$ be the linearization of $\psi,$ so $h=\psi|_{m/m^2}\in GL(V).$ 
  Then we have $h\gamma h^{-1}=\sigma(\gamma), \gamma\in\Gamma.$ Thus we can extend $h$ to an automorphism of $\Sym(\hat{V})\#\Gamma.$
So, $K_0(h)$ is a diagram automorphism.
Now, we claim that $K_0(f)=K_0(h).$ Indeed, let $t=fh^{-1}\in\Aut(\Sym(\hat{V})\#\Gamma).$
Then $t$ is a $\Gamma$-invariant automorphism of $\Sym(\hat{V})\#\Gamma.$ Hence $K_0(t)=\Id$ and $K_0(f)$ is a diagram automorphism as desired.

Finally, to finish the proof, it suffices to show that given $\theta\in\Pic(\Sym(\hat{V})\#\Gamma)$ such that
it acts as the identity on the center $\mathcal{O}(\hat{V}/\Gamma)$, then $K_0(\theta)$ is a diagram automorphism.
Let us put $A=\Sym(\hat{V})\#\Gamma, Z=Z(A)$ for simplicity.
Denote by $\Pic(A)^{id}$ the kernel of the restriction homomorphism on the center $\Pic(A)\to \Aut(Z(A)).$
As a first step, we show that there is a natural injective homomorphism $\Pic(A)^{id}\to \Pic(U),$ where
recall that $U$ denotes the nonsingular locus of $\hat{V}/\Gamma.$
Put $M=\theta(A)$, so $M$ is an invertible $A$-bimodule.
By the assumption, $M|_U$ is a module over $A\otimes_ZA^{op}|_U.$ So, since $A|_U$ is an Azumaya algebra,
it follows that $M|_U=A|_{U}\otimes L$, where $L$ is a coherent sheaf on $U$. But, since $M|_U$ is an invertible bimodule over $A|_U$, it follows
that $L$ must be a line bundle. Now, recall that $\Pic(U)$ can be identified with the character group of $\Gamma.$ Indeed,
given $\chi: \Gamma\to\bold{k}^*$ then $\bold{k}^*\times_{\Gamma} W$ is naturally a line bundle on $U$ (here recall that $W=\pi^{-1}(U),$
and since $\hat{V}\setminus{U}$ has codimension $\geq 2$ in $\hat{V}$ (and $\Pic(\hat{V})$ is trivial).
So, we have the desired restriction homomorphism $$\Pic(A)^{id}\to \Pic(U)\cong\Gamma^*.$$ In fact, we next show that this is an isomorphism.
This statement is a slight generalization of [\cite{Lo}, Lemma 2.1]
 Indeed, if 
$M_U\cong A|_U$ then since $M, A$ are both Cohen-Macaulay $Z$-modules (since $M$ is a projective left (and right) $A$-module),
it follows that $M=\Gamma(U, M|_{U})$ and
 $A=\Gamma(U, A|_U).$
So, $M=A$ as an $A$-bimodule. 
To show the surjectivity, let $L\in\Pic(U)$ and $\chi\in\Gamma^*$ be the corresponding character.
Then we have an $\mathcal{O}(V/\Gamma)$-linear automorphism $\tau_{\chi}\Aut(A)$ defined as follows
$$\tau(\gamma)=\chi(\gamma)\gamma,\gamma\in\Gamma, \quad\tau(v)=v, v\in A(V).$$
Now, viewing $\tau_{\chi}$ as an element of $\Pic(A)$, we see that its image under the restriction homomorphism is $L.$
Recall that by the definition, $K_0(\tau_{\chi})$ is a diagram automorphism. Hence we are done.

\end{proof}

Recall that by a well-known result of Stafford  \cite{S}, $\Pic(A_1(\bold{k}))=\Aut(A_1(\bold{k}))$ for any algebraically closed field $\bold{k}.$
Next result shows that this is no longer the case for many fixed rings of $A_n(\bold{k}).$

\begin{prop}
Under the above notation/assumptions
$\Pic(A(V)^{\Gamma})/\Aut(A(V)^{\Gamma})$ contains $\Gamma^*.$ 

\end{prop}
\begin{proof}
Let $\chi\in\Gamma^*.$ Then recall that we have the corresponding twisting automorphism
$\tau_{\chi}\in \Aut(A(V)\# \Gamma)$ which act a the identity on $A(V).$ Since $\Pic(A(V)\#\Gamma)=\Pic(A(V)^{\Gamma})$, we
have  $\tau_{\chi}\in \Pic(A(V)^{\Gamma}).$ Moreover, $Z(\tau_{\chi})$ is trivial. Thus, if
$\tau_{\chi}$ is given by some $\theta\in\Aut(A(V)^{\Gamma})$, it follows by the Noether-Skolem theorem that there
exists $a\in A(V)^{\Gamma},$ such that $$\theta(x)=axa^{-1}, \quad x\in A(V)^{\Gamma}.$$
But, since $\theta$ has a finite order, there exists a root of unity  $1\neq \alpha\in\bold{k}$ and $0\neq x\in A(V)^{\Gamma}$
so that $\theta(x)=\alpha x.$ Hence $ax=\alpha xa$ which easily leads to  a contradiction if we take the associated graded of both sides with respect
to the standard filtration of $A(V)^{\Gamma}.$

\end{proof}




\section{The proof of the main result}

We make crucial use of the following well-known result. For the convenience of the reader we include a short proof kindly communicated to us by P. Etingof.

\begin{theorem}\label{Chebotarev}
Let $S\subset\mathbb{C}$ be  a finitely generated $\mathbb{Z}$-algebra. Then there exists infinitely many primes $p$ for which there
exists a ring homomorphism from $S$ to $\mathbb{F}_p.$

\end{theorem}
\begin{proof}
Replacing $S$ by its image under any ring homomorphism $S\to\overline{\mathbb{Q}}$,
we may assume that $S\subset \overline{\mathbb{Q}}.$ Let $K$ denote the field of
fractions of $S.$ By the primitive element theorem,
there exists $t\in S$ so that $K=\mathbb{Q}(t).$
Let $f\in\mathbb{Q}[x]$ be the minimal polynomial of $t$ over $\mathbb{Q}.$
So, we have an embedding of $S$ in $\mathbb{Z}_l[x]/(f)$ for some $l\in\mathbb{N}.$ 
Then given a prime $p>l$ such that $f$ has  a root
in $\mathbb{F}_p$, then we get a nontrivial homomorphism from S to $\mathbb{F}_p.$ 
Hence to finish the proof, it suffices to show that given a nonconstant polynomial 
$g\in\mathbb{Z}[x],$
there exists infinitely many primes $p$ so that $g$ has a root in $\mathbb{F}_p.$

Indeed, assume the contrary, hence there exist finitely many primes $p_1,\cdots, p_k$ so that
for any $n\in\mathbb{Z}, g(n)$  has no prime divisor outside $p_1,\cdots, p_k.$ 
Without loss of generality $g(0)\neq 0$. Then put $h(x)=g(xg(0))/g(0)\in\mathbb{Z}[x]$. Then 
$h(0)=1, \deg(h)=\deg(g)$ and
for any $n\in\mathbb{Z}$ all prime divisors of $h(n)$ belong to $\lbrace p_1,\cdots, p_k\rbrace.$
Then $h(tN)=1 \mod N, N=\prod_{i=1}^kp_i$ for all $t\in\mathbb{Z}$, hence $h(tN)=\pm 1,$ which is a contradiction.
\end{proof}

We next recall couple of standard facts about Morita equivalences. If 
$$F: A-mod\to B-mod$$ is a Morita equivalence of flat $\mathbb{Z}$-algebras,
then we get a Poisson isomorphism on the level of centers $Z(F): Z(A/pA)\cong Z(B/pB)$, where the Poisson bracket on the center of $A/pA$
is the standard one: if $a_1, a_2\in A$ are so that $a_1\mod p\in Z(A/pA)$ and $a_2\mod p\in Z(A/pA)$ then their Poisson bracket is defined as follows:
$$\lbrace \bar{a_1}, \bar{a_2}\rbrace=(\frac{1}{p}[a_1, a_2])\mod p.$$
Also, the Poisson bracket on $Z(A(V/pV))$ coincides with the negative of the standard one after the identification
of the center with the Frobenius twist of $\mathcal{O}(V_{\mathbb{F}_p})$ (see \cite{BK}).  Thus, a similar statement holds
for the Poisson bracket on the center of $A(V/pV)\#\Gamma.$
\begin{proof}[Proof of Theorem \ref{main}]
Denote by $I$ the set of isomorphism classes of irreducible $\Gamma$-modules.
As before, we will identify $K_0(H_\lambda)$ (for any $\lambda$) with  $K_0(\mathbb{C}[\Gamma])=\mathbb{Z}^I$. 
Let $$F:H_c-mod\to H_{c'}-mod$$
be a Morita equivalence.
It follows that there exists a finitely generated ring
$S\subset\mathbb{C}$, containing entries of $c, c',$ such
that $F$ is a $S$-linear Morita equivalence.
Making $S$ large enough so that $K_0(S[\Gamma])=\mathbb{Z}^I$,
we get an isomorphism of the Grothendieck groups $$K_0(F): K_0(H_{c})=\mathbb{Z}^I\cong K_0(H_{c'})=\mathbb{Z}^I.$$
Then by using Proposition \ref{Mitya}, we get that for any homomorphism $\chi: S\to\mathbb{F}_p$ (for $p\gg 0$) there exist (recall that $H_0=A(V)\#\Gamma)$ 
$w, w'\in W_{\text{ext}}\subset \Aut(\mathbb{Z}^I)$ 
and  Morita equivalences  
$$G_1: H_{\chi(c)}-mod\to H_{0}-mod,\quad G_2:H_{\chi(c')}-mod\to H_{0}-mod,$$
such that $K_0(G_1)=w$ and $K_0(G_2)=w'.$
Thus, combining above equivalences we have the following auto-equivalence
$$H=G_1F_{\chi}G_2:A(V)\#\Gamma-mod\cong A(V)\# \Gamma-mod.$$
Moreover, on the level of the center $H$ preserves the Poisson bracket.
Now, by Theorem \ref{Mitya} we can conclude that $K_0(H)\in \Aut(\mathbb{Z}^{I})$ is a diagram automorophism.
So, there exists a diagram automorphism $\rho$ and $w\in W_{ext}$ (depending on $\chi)$, such that $$K_0(F)=w\rho\in W_{ext},\quad \chi(c')=\rho w(\chi(c)).$$ 
Now, assume that $\lambda'\neq \rho K_0(F)(\lambda)$ for any diagram automorphism $\rho.$ 
Then since the number of diagram automorphisms is finite, using Theorem \ref{Chebotarev} we
may conclude that there are infinitely many primes $p$ for which there exists a homomorphism $\chi:S\to\mathbb{F}_p$, so that
$\chi(c')\neq\rho K_0(F)\chi(c)$ for any diagram automorphism $\rho$, which is a contradiction.
\end{proof}

\section{Application to generalized Weyl algebras}

In this section we use Theorem \ref{key} to solve the Morita equivalence problem
for noncommutative deformationas of type A  Kleinian singularities, also known as (classical) generalized Weyl algebras. In order to do so, we need to compute the action
of reflections and diagram automorphisms in terms of parameters of generalized Weyl algebras.

Recall that we want to show that given $v, v'\in\mathbb{C}[h]$ polynomials of degree $m$ with no multiple roots,
then $A(v)$ is Morita equivalent of $A(v')$ if and only if up to a permutation and a sign roots of $v$ and $v'$ differ by integers (Theorem \ref{GWA}).
It is more convenient to work with roots of $v, v'$, thus we put $v=\prod_{i=0}^{m-1} (h-t_i), v'=\prod_{i=0}^{m-1}(h-t'_i)$ (since without loss
of generality $v, v'$ are monic)
and we denote $A(v)$ by $A_t,$ (respectively $A(v')$ by $A_{t'}.$)

Denote by $G\leq \Aut(\mathbb{C}^m)$ the subgroup generated by translations by $\mathbb{Z}^m$, translations by the diagonal, permutations and
the sign change $v\to-v.$ Thus, to show the theorem, we need to prove that if
$A_t$ and $A_{t'}$ are Morita equivalent, then $t'=\gamma(t)$ for some $\gamma\in G.$
We start by recalling the dictionary between generalized Weyl algebras $A_t$ and deformed preprojective algebras
of the extended Dynkin diagram of type A.

Without loss of generality, we may assume that both $t, t'$ have 1 as the last coordinates.

From now on, let $Q$ denote the quiver corresponding to the extended
Dynkin diagram $\tilde{A}_m$, which is a cycle with $m$ vertices. 
(This corresponds to the cyclic group $\Gamma$ of order $m$.)
Thus we have an arrow $e_i\to e_{i+1}$ for $i\in\mathbb{Z}/m\mathbb{Z}.$
So, the Ringel form on $Q$ on coordinate vectors are given by 
$$(\epsilon_i, \epsilon_{i\pm 1})=-1, (\epsilon_i, \epsilon_i)=2.$$
The corresponding reflections are given by the following formulas
$$s_i: \mathbb{Z}^m\to\mathbb{Z}^m: (s_i(\alpha))_i=\alpha_{i+1}+\alpha_{i-1}-\alpha_i, s_i(\alpha)_j=\alpha_j, i\neq j.$$
$$r_i: \mathbb{C}^m\to \mathbb{C}^m: (r_i(\lambda))_i=-\lambda_i, r_i(\lambda)_{i\pm 1}=\lambda_i+\lambda_{i\pm 1}.$$
We have that $\delta=(1,\cdots, 1).$

It is well-known that a generalized Weyl algebra $A_t$ can be identified
with $\mathcal{O}^{\lambda}=e_0\Pi^{\lambda}(Q)e_0$ for an appropriate
$\lambda.$
Indeed, let $\lambda=(\lambda_0,\cdots,
\lambda_{n-1})\in\mathbb{C}^{n},$ so that $\lambda\cdot\delta=1.$ Define 
$t({\lambda})\in\mathbb{C}^m$ as follows
 $$t_i(\lambda)=\sum_{j\leq i}\lambda_j,\quad \lambda_i=t_i-t_{i-1}.$$
 Then we have that $\mathcal{O}^{\lambda}\cong A_{t(\lambda)}$ (see for example \cite{M}.)
A direct computation shows the following:
$$t_j({r_i(\lambda)})=t_j(\lambda), j\notin\lbrace i-1, i\rbrace, t_{i-1}(r_i(\lambda))=t_i, t_{i}(r_i(\lambda))=t_{i-1},\quad 0<i<m-1,$$
$$t_0(r_0(\lambda))=-t_0, t_i(r_0(\lambda))=t_i-t_0, i>0,$$
$$t(r_{m-1}(\lambda))=(t_0+1-t_{m-2},\cdots, t_{m-3}+1-t_{m-2}, 2-t_{m-2}, 1).$$

Denote by $\rho$ and $\tau$ the following diagram automorphisms of $Q$
$$\rho(\lambda)=(\lambda_1. \lambda_2, \cdots, \lambda_{m-1}, \lambda_0), \quad\tau(\lambda)=(\lambda_{m-1}, \lambda_{m-2}, \cdots, \lambda_0).$$
Then
$$t(\rho(\lambda))=(t_1-t_0, t_2-t_0,\cdots, t_{m-2}-t_0, 1-t_0, 1), $$
$$t(\tau(\lambda))=(1-t_{m-2}, 1-t_{m-3},\cdots, 1-t_0, 1).$$
Recall that $W_{ext}$ denotes the subgroup of automorphisms of $\mathbb{C}^m$ generated by diagram automorphisms and reflections corresponding to the quiver $Q_m.$
Now, it follows easily from the above formulas that for any $\gamma\in W_{ext}$, we have that $t(\gamma(\lambda))\in G
t(\lambda).$ Put $t=t(\lambda), t'=t(\lambda').$ Hence $H_{\lambda}$ and
$H_{\lambda'}$ are Morita equivalent. Therefore by Theorem \ref{main} there exists $\gamma\in W_{ext}$ such that $\lambda'=\gamma(\lambda).$
So $t'\in G t$ and we are done.
\begin{remark}

In the case of a generic parameter $\lambda$ (when gl.dim$(\mathcal{O}^{\lambda})=1$), there is a classification result of
Baranovsky, Ginzburg and Kuznetsov  \cite{BGK} of isomorphism classes of indecomposable projective $\mathcal{O}^{\lambda}$-modules in terms
of certain quiver varieties. This bijection was upgraded to an $\Aut(A(v))$-equivariant bijection by F.Eshmatov \cite{E} in the case of type A singularities.
Combining these results with a transitivity of the action of $\Aut(A(v))$ on those quiver varieties (proved in \cite{CEET}), theorem \ref{GWA} can be deduced
for generic parameters (see \cite{T}).
An advantage of the proof in this paper is that we do not need to utilise above mentioned nontrivial results.

\end{remark}

\section{Rational Cherednik algebras of type A}

In this section we apply the approach in the previous ones to rational Cherednik algebras of type A. 
Let us recall the precise definition. Let $\mathfrak{h}\subset\mathbb{C}^n$ be the standard $n-1$-dimensional representation of $S_n.$
Let $t\in\mathbb{C}.$ Then the corresponding rational Cherednik algebra $H_t, t\in\mathbb{C}$ is defined a symplectic reflection algebra
of the pair $(S_n, \mathfrak{h}\oplus\mathfrak{h}^*).$ In this case the space of parameters is one dimensional.
In regards to the Morita classification problem for $H_c$, it was proved by Berest, Etingof, and  Ginzburg \cite{BEG} that if 
$c$ is transcendental then $H_c$ is Morita equivalent to $H_{c'}$ (for any $c'$)
if and only if $c'\in \pm c+\mathbb{Z}.$ They have also explained that if $H_c, H_{c'}$ are Morita equivalent, then $\mathbb{Q}(c)=\mathbb{Q}(c').$

Just as before, translation functors play the
key role. Let us start by recalling the Heckman-Opdam shift  (translation) functors for $H_c.$ Recall once again that $B_c=eH_ce,$ where $e$ is the symmetrizer idempotent.
Recall that a parameter $c$ is called spherical if $H_c=H_ceH_c$, equivalently if $H_c$ is Morita equivalent to $B_c.$ 
It was proved in \cite{BEG} that $B_c\cong B_{-1-c}$ for all $c.$
Combining this with the isomorphism $H_c\cong H_{-c}$ given by the twist by the sign character
of $S_n$ and the Morita equivalence between $H_c$ and $eH_ce$  we
obtain the desired translation equivalence (for generic enough $c$) $$S_c: H_c-mod\cong H_{c+1}-mod.$$
Denote by $I$ the set of isomorphism classes of $S_n$-modules. So, we may identify
$K_0(H_c)$ with $K_0(\mathbb{C}[S_n])=\mathbb{Z}^I$ in the standard manner.
We recall that on the level of $K$-groups 
$$K_0(S_c): K_0(H_c)=\mathbb{Z}^I\to K_0(H_{c+1})=\mathbb{Z}^I$$ is independent of
$c$ and we denoted it by $S\in\Aut_{\mathbb{Z}}(\mathbb{Z}^I).$

Next we view $H_c$ as a $\mathbb{Z}[c]$-algebra.
It is known that \cite{BFG} $H_c/H_ceH_c$ is a finite torsion $\mathbb{Z}_l[c]$-module (for some $l>1)$.
Hence, it follows that $H_c/H_ceH_c$ is annihilated by some $g\in \mathbb{Z}[c]$ and roots of $g$ are precisely
aspherical values of $H_c$. Now we recall that the following is the set of all aspherical values of $H_c$ \cite{BE}
$$Q_n=\lbrace -\frac{i}{d}, 1\le i<d, 2\leq d\leq n\rbrace.$$

Let $h=\prod_{\alpha\in Q_n}(c-\alpha)^t$ for large enough $t.$ So, $h\in Ann(H_c/H_ceH_c).$
Thus, we conclude that the following holds.
For all $p\gg 0$ and any $t, t'\in [0, p-1]$ that belong to the same connected component of $[0, p)\setminus (Q_n\mod p),$
where elements of $Q_n \mod p$ are identified with integers in $[0, p),$ then we have the translation
Morita equivalence over $\mathbb{F}_p$ $$S_{t}^{t'}: H_t-mod\to H_{t'}-mod.$$
Moreover $K_0(S_{t}^{t'})=S^{t'-t}.$

\begin{theorem}

Let $c=\frac{a}{l}, c'=\frac{a'}{l'}$ be a pair of rational numbers in reduced form such that
$(ll', n!)=1.$ Assume that either $(l, l')=1$ or $l=l'$ and there exists an integer $\alpha$ so that $a'=\alpha a\mod l, l< n\alpha.$  
Then $H_c$ is Morita equivalent to $H_{c'}$ if and only if $c'\in \pm c+\mathbb{Z}.$
In particular, if $c=\frac{a}{l}$ with $(l, n!)=1,$ then $H_c$ is Morita equivalent to $H_m$ for any $m\in\mathbb{Z}$ if and only if $c\in\mathbb{Z}.$

\end{theorem}
\begin{proof}

Recall that we have a polynomial $h=\prod_{\alpha\in Q_n}(c-\alpha)^t$
with a property that for any $p\gg 0$ and $c\in\mathbb{F}_p$
such that $h(c)\neq 0,$ then $H_c$ and $H_{c+1}$ are Morita equivalent via the standard translation functor.

Suppose that $c'$ and $c$ have the same denominator $l$ and $(l, n!)=1.$ Also, $c'=\alpha c+\beta$ for integers $\alpha, \beta,$ such that $l<\alpha n.$
Let $p\gg 0$ be a prime so that $$p=1\mod n!, \quad p=-a \mod l.$$ 
There are infinitely many such primes $p$ by the Dirichlet theorem. Given $x\in\mathbb{Q}$, denote by $\bar{x}$ its image in
$\mathbb{F}_p.$
Then $$\bar{c}=\frac{p+a}{l},\quad \bar{c'}=\alpha\frac{p+a}{l}+\beta.$$ 
So $\bar{c}, \bar{c'}\in [0, \frac{p-1}{n}).$ In particular, $0, \bar{c}, \bar{c'}$ belong to the same
connected component of  $[0, p-1)\setminus Q_n\mod p.$
Thus, both $H_{\bar{c}}, H_{\bar{c'}}$ are Morita equivalent with $H_0=A(V)\# S_n$ using the standard translation functors:
$$S_0^{\bar{c}}:H_0-\text{mod}\to H_{\bar{c'}}-\text{mod}, \quad S_{0}^{\bar{c'}}: H_0-\text{mod}\to H_{\bar{c'}}-\text{mod}.$$
Let $$F: H_c-mod\to H_{c'}-mod$$ be a Morita equivalence.
So, $F$ is defined over a sufficiently large finitely generated $\mathbb{Z}$-subalgebra $R$ of $\mathbb{C}.$ Then for all $p\gg 0$, there exists a base change 
$R\to \bar{\mathbb{F}_p}$ and the corresponding
Morita equivalence $$\bar{F}:H_{\bar{c}}-mod\to H_{\bar{c'}}-mod.$$
Moreover $K_0(F)=K_0(\bar{F})\in \Aut(Z^I)$  is independent of $p.$
Thus, combining above equivalences, we get that
$$S_0^{\bar{c}}\bar{F}S_{\bar{c'}}^{0}\in\Pic(A(V)\#S_n).$$
Denote by $\epsilon\in \Aut(\mathbb{Z}^I)$ the element corresponding to the action on $K_0$
by the automorphism induced by
the sign character of $S_n.$ 
So, by using Theorem \ref{key} we may conclude that
$K_0(F)=\epsilon^i S^{\bar{c}-\bar{c'}},$ where $i$ is 0 or 1.
So $\bar{c}-\bar{c'}=\pm t$ for a fixed $t\in\mathbb{Z}.$ Hence $c'\in\pm c+\mathbb{Z}$ as desired.

Now suppose $(l, l')=1.$ Let $p\gg 0$ be a prime so that $$p=-a\mod l,\quad -a'\mod l',\quad 1\mod n!.$$ 
Then $$\bar{c}=\frac{p+a}{l},\quad \bar{c'}=\frac{p+a'}{l'},\quad \bar{\frac{i}{d}}=i\frac{p+1}{d},$$
So, just like in the previous case, $\bar{c}, \bar{c'}\in [0, \frac{p-1}{n})$ and $H_c, H_{c'}$ are Morita equivalent via the translation functors,
and the similar argument as above applies.
\end{proof}



\end{document}